\documentclass[reqno,11pt]{amsart}

\usepackage{amsmath,amssymb,amsthm,mathtools}
\usepackage{mathrsfs}       
\usepackage{hyperref}
\usepackage{xcolor}         
\usepackage{esint}
\usepackage{tikz}           
\usepackage{verbatim}       
\usepackage{setspace}       
\usepackage[normalem]{ulem}
\usepackage{cancel}

\title[Lipschitz regularity for the parabolic $(s,p)-$obstacle problem]{Lipschitz regularity for the parabolic $(s,p)-$obstacle problem}

\author[D. Jesus]{David Jesus}
\address{Applied Mathematics and Computational Sciences (AMCS), Computer, Electrical and Mathematical Sciences and Engineering Division (CEMSE), King Abdullah University of Science and Technology (KAUST), Thuwal, 23955-6900, Kingdom of Saudi Arabia}
\email{david.dejesus@kaust.edu.sa}

\author[A. Sobral]{Aelson Sobral}
\address{Applied Mathematics and Computational Sciences (AMCS), Computer, Electrical and Mathematical Sciences and Engineering Division (CEMSE), King Abdullah University of Science and Technology (KAUST), Thuwal, 23955-6900, Kingdom of Saudi Arabia}
\email{aelson.sobral@kaust.edu.sa}

\author[J.M. Urbano]{Jos\'{e} Miguel Urbano}
\address{Applied Mathematics and Computational Sciences (AMCS), Computer, Electrical and Mathematical Sciences and Engineering Division (CEMSE), King Abdullah University of Science and Technology (KAUST), Thuwal, 23955-6900, Kingdom of Saudi Arabia and CMUC, Department of Mathematics, University of Coimbra, 3000-143 Coimbra, Portugal}
\email{miguel.urbano@kaust.edu.sa}

\newtheorem{theorem}{Theorem}[section]
\newtheorem{lemma}{Lemma}[section]

\newtheorem{proposition}{Proposition}[section]
\newtheorem{definition}{Definition}[section]

\newcommand{\Tail}{\operatorname{Tail}}

\newcommand{\abs}[1]{\left|#1\right|}

\newcommand{\R}{\mathbb{R}}

\newcommand{\p}{\partial}
\newcommand{\LL}{\mathcal{L}}
\newcommand{\II}{\mathcal{I}}

\newcommand{\dd}{\, \mathrm{d}}

\newcommand{\C}{\mathcal{C}}

\newcommand{\intav}[1]{\mathchoice 
  {\mathop{\vrule width 6pt height 3 pt depth -2.5pt \kern -8pt \intop}\nolimits_{\kern -6pt#1}} 
  {\mathop{\vrule width 5pt height 3 pt depth -2.6pt \kern -6pt \intop}\nolimits_{#1}}
  {\mathop{\vrule width 5pt height 3 pt depth -2.6pt \kern -6pt \intop}\nolimits_{#1}}
  {\mathop{\vrule width 5pt height 3 pt depth -2.6pt \kern -6pt \intop}\nolimits_{#1}}}

\numberwithin{equation}{section}
\begin{document}








\begin{abstract} 
We study the obstacle problem for the parabolic fractional \(p-\)Laplace equation  \[\partial_t u+(-\Delta_p)^su = 0\] in the degenerate range \(2<p<2/(1-s)\). We prove that viscosity solutions are locally Lipschitz continuous in space and H\"older continuous in time. If, in addition, \(p>1/(1-s)\), the time regularity improves to Lipschitz continuity.
\end{abstract}  

\date{\today}

\maketitle

\section{Introduction}

The obstacle problem is arguably one of the most emblematic free boundary problems, with applications ranging from the equilibrium of elastic membranes constrained by an obstacle and Signorini-type contact problems, to Stefan-type phase transitions, dam/filtration problems in porous media, and optimal stopping problems in mathematical finance such as American option pricing; see, for instance, \cite{Rodrigues1987} or \cite[Section~3]{RosOton2018}.

In this note, we study for the first time the parabolic obstacle problem driven by the fractional \(p-\)Laplacian. For \(Q_1=B_1\times(-1,0]\), \(s\in(0,1)\), and \(p>2\), given an obstacle \(\varphi\colon Q_1\to\mathbb{R}\), we consider functions \(u\) satisfying
\begin{equation}\label{eq:parabolic-obstacle-problem}
    \min\left\{u-\varphi,\,\partial_t u+(-\Delta_p)^s u\right\} = 0 \qquad \text{in } Q_1 .
\end{equation}
Here, for each fixed \(t\), the operator is defined by
\[
    (-\Delta_p)^s u(x,t)
    \coloneqq
    \operatorname{P.V.}\int_{\mathbb{R}^d}\frac{|u(x,t)-u(y,t)|^{p-2}\bigl(u(x,t)-u(y,t)\bigr)}{|x-y|^{d+sp}}\,\dd y .
\]
Equivalently, \eqref{eq:parabolic-obstacle-problem} corresponds to
\[
    u\geq \varphi,\qquad \partial_t u+(-\Delta_p)^s u\geq 0 \qquad \text{in } Q_1,
\]
and
\[
    \partial_t u+(-\Delta_p)^s u=0 \qquad \text{in } \{u>\varphi\}\cap Q_1.
\]
Thus \(u\) evolves according to the fractional \(p-\)caloric equation away from the contact set, while remaining constrained to lie above the obstacle.

For \(p=2\), \eqref{eq:parabolic-obstacle-problem} reduces to the
parabolic fractional obstacle problem studied in \cite{CF, BFR}; related obstacle problems also arise, for instance, in option-pricing models \cite{BM}. In the nonlinear nonlocal setting, the elliptic obstacle problem for integro-differential operators with fractional \(p-\)growth was studied in \cite{KKP0}. Away from the contact set, the constraint is inactive, and the solution is governed by the unconstrained parabolic fractional \(p-\)Laplacian flow. Regularity properties of this evolution equation have been investigated in \cite{AMRT, S, BLS2, L, GLT, JSU26}.

The problem also admits a natural variational interpretation. Namely, \eqref{eq:parabolic-obstacle-problem} can be viewed as the gradient flow of the fractional \(p-\)Dirichlet energy
\[
    v \mapsto \frac{1}{p}[v]_{W^{s,p}}^{p}
\]
over the closed convex set
\[
    \mathcal{K}_{\varphi}
    \coloneqq
    \left\{ v\in W^{s,p} \colon v\geq \varphi \text{ a.e.} \right\}.
\]
This naturally leads to weak solutions rather than viscosity solutions. In the elliptic setting, the weak theory for equations driven by the fractional \(p-\)Laplacian is by now well developed; see, for example, \cite{KMS2, KMS3, KMS1, DFM}. 

Our goal is to establish foundational regularity estimates for solutions of the parabolic \((s,p)-\)obstacle problem. For obstacles $\varphi$ satisfying the suitable regularity condition
\[
    \varphi\in C_x^{0,1}(Q_1)\cap C_t^{0,\alpha_\varphi}(Q_1), \quad \text{ for } \alpha_\varphi \in (0,1],
\]
we show that, under a mild tail-continuity assumption, solutions are Lipschitz continuous in space and H\"older continuous in time. We work on parabolic cylinders of the type $Q_r \coloneqq B_r(0)\times (-r^2,0]$.

\begin{theorem}\label{t:Lipschitz}
Let $p > 2$ and $s \in (0,1)$ be such that
\[
    q_c \coloneqq - 1 + p(1-s) < 1.
\]
Let \(u\) be a viscosity solution to \eqref{eq:parabolic-obstacle-problem} in \(Q_1\), and assume that
\begin{equation}\label{eq:LiDi_fora}
    u \in C_{\mathrm{loc}} (-1,0;L^{p-1}_{sp}(\R^d)).
\end{equation}
Then, for $\alpha$ given by
\begin{align*}
    \alpha = \min\left\{ \left( \dfrac{1}{1 - q_c}\right)^-,\alpha_\varphi \right\},
\end{align*}
there exists a constant \(C>0\), depending only on \(d\), \(p\), \(s\), \(\alpha\), $[\varphi]_{C_x^{0,1}}$, and $ [\varphi]_{C^{0,\alpha_\varphi}_t}$ such that
\[
    |u(x,t)-u(y,\tau)| \leq C \mathcal{M} \Bigl(|x-y| + \mathcal{M}^{(p-2)\alpha}|t-\tau|^\alpha\Bigr),
\]
for all \((x,t),(y,\tau) \in Q_{1/2}\), where
\begin{equation}\label{eq:constant-M}
    \mathcal{M} \coloneqq 1+\|u\|_{L^\infty(Q_{3/4})} + \sup_{t \in (-(3/4)^2,0]} \|u(\cdot,t)\|_{L^{p-1}_{sp}(\R^d)}.
\end{equation}

In particular, if \(q_c>0\) and $\alpha_\varphi=1$, then 
\[
    |u(x,t)-u(y,\tau)| \leq C \mathcal{M} \Bigl(|x-y| + \mathcal{M}^{p-2}|t-\tau|\Bigr),
\]
for all \((x,t),(y,\tau) \in Q_{1/2}\).
\end{theorem}

A few remarks are in order. First, the assumption that the obstacle is Lipschitz continuous in space is imposed mainly to present the proof in a clean form. If instead
\[
    \varphi\in C_x^{0,\beta_\varphi}(Q_1),
\]
for some \(\beta_\varphi<1\), then the proof of Theorem \ref{t:lip-in-space_Chibrata} gives
\[
    u\in C_x^{0,\beta_\varphi}(Q_{1/2}),
\]
by stopping the spatial bootstrapping argument at the exponent \(\beta_\varphi\). Since the time regularity is obtained independently, the same argument still gives spatial regularity when \(\alpha_\varphi=0\), although in that case it does not yield any improvement in time.

Second, the tail-continuity assumption \eqref{eq:LiDi_fora} is not needed to obtain spatial regularity of \(u\) with any exponent
\[
    \gamma < \min\left\{1,\frac{sp}{p-1}\right\}.
\]
Under the corresponding regularity assumptions on the obstacle, the same method also yields analogous time regularity.

The spatial scale appearing here is consistent with the elliptic regularity theory for fractional \(p-\)Laplace type equations, where H\"older, Lipschitz, and higher regularity estimates have been obtained in different settings; see, for instance, \cite{BT, BS, GJS, JSU26}.

We now clarify the mechanism behind Theorem \ref{t:Lipschitz}. At a heuristic level, one expects such an estimate to hold for the following reason. In the non-coincidence set \(\{u>\varphi\}\), the solution is fractional \(p-\)caloric and therefore enjoys the regularity of solutions to the homogeneous equation. On the coincidence set
\[
    \Lambda_\varphi(u) \coloneqq \{u=\varphi\},
\]
the solution is constrained by the obstacle and should inherit its regularity. The main difficulty is to turn this informal dichotomy into a quantitative estimate that remains valid across the free boundary.

We address this issue within the viscosity framework, combining a localization argument with the Ishii--Lions doubling variables method. The nonlocal version of the Ishii--Lions approach was first developed for the elliptic nonlocal $p-$Laplacian in \cite{BT}. We refer to \cite{CIL} for the classical viscosity theory and to \cite{IN, CLD, IJS} for related nonlocal and parabolic viscosity techniques. The penalization is chosen so that the test functions produced by the argument have large spatial gradients. This forces the maximum point on the subsolution side to lie away from the obstacle, where the equation is satisfied in the usual viscosity sense. The maximum point on the supersolution side may still belong to the coincidence set, but this causes no difficulty, since the viscosity supersolution inequality remains available there.

A further difference with respect to \cite{JSU26} concerns the treatment of the time variable. In that work, the doubling variables argument uses the H\"older continuity in time obtained from the weak theory. In the present viscosity setting, no such regularity theory is available in the literature. However, the argument does not require a quantitative H\"older estimate at this stage: it is enough to use the modulus of continuity in time of \(u\), which follows from the viscosity formulation. This modulus allows us to absorb the error terms produced by the localization in time. Once this point is addressed, the rest of the proof for the regularity in the space variable follows the strategy developed in \cite{JSU26}.

With the spatial regularity estimate in hand, we then turn to the regularity in time. As in \cite{JSU26}, the argument uses the already established spatial regularity to construct a suitable barrier. In the present obstacle setting, however, this step requires some additional care. The barrier must be a subsolution and, at the same time, must remain strictly above the obstacle in the relevant region. This is achieved through a further localization argument, which uses the assumed time regularity of the obstacle.  We then prove a comparison principle for the parabolic \((s,p)-\)obstacle problem, allowing the comparison with the barrier to be propagated into the interior. The comparison argument also contains a minor localization subtlety. In the contradiction argument, the penalization pushes the maximum point on the subsolution side away from the obstacle, so that the equation is satisfied there in the usual viscosity sense. Once this is ensured, the proof proceeds as in \cite{JSU26}.

The paper is organized as follows. In Section~\ref{s:prelim}, we introduce the viscosity formulation of the parabolic \((s,p)-\)obstacle problem and prove the comparison principle. In Section~\ref{s:IL}, we establish the spatial regularity estimate by adapting the Ishii--Lions argument described above. In Section~\ref{s:time:reg}, we prove the regularity in time and complete the proof of Theorem~\ref{t:Lipschitz}.

\section{Viscosity solutions and the comparison principle}\label{s:prelim}

We first fix some notation. Throughout the paper, \(\Omega\subset\R^d\) denotes a bounded smooth domain. We write
\[
    Q_r(x_0,t_0) \coloneqq B_r(x_0)\times (t_0-r^2,t_0],
    \qquad
    Q_r \coloneqq B_r(0)\times (-r^2,0].
\]

For \(t_1<t_2\), we denote by \(C^{2}(\Omega\times(t_1,t_2])\) the class of functions \(v\colon\Omega\times(t_1,t_2]\to\mathbb R\) that are \(C^2\) in the spatial variables and \(C^1\) in time. Similarly, \(C^{1,1}(\Omega\times(t_1,t_2])\) denotes the class of functions that are \(C^{1,1}\) in space and \(C^{0,1}\) in time.

We also use the tail space \(L^{p-1}_{sp}(\R^d)\), consisting of all functions
\(w\in L^{p-1}_{\mathrm{loc}}(\R^d)\) such that
\[
    \|w\|_{L^{p-1}_{sp}(\R^d)}
    \coloneqq
    \left(
    \int_{\R^d}
    \frac{|w(x)|^{p-1}}{1+|x|^{d+sp}}\,\dd x
    \right)^{\frac{1}{p-1}}
    <\infty.
\]

We write \(\gamma^-\) for a fixed number strictly smaller than \(\gamma\), and \(\gamma^+\) for a fixed number strictly larger than \(\gamma\). Unless otherwise stated, a constant is called universal if it depends only on \(d\), \(p\), and \(s\). Small universal constants are denoted by \(c\), and large universal constants by \(C\); their values may change from line to line. When a constant has to remain fixed throughout an argument, we enumerate it.
 
We next introduce the notion of viscosity solution to \eqref{eq:parabolic-obstacle-problem}.

\begin{definition}\label{def:viscosity-solution-obstacle}
Let \(\varphi\colon \Omega\times(t_1,t_2]\to\mathbb R\) be continuous. We say that \(u\) is a viscosity subsolution to the parabolic obstacle problem
\[
    \min\left\{
        u-\varphi,\,
        \partial_t u+(-\Delta_p)^s u
    \right\}=0
    \qquad\text{in }\Omega\times(t_1,t_2]
\]
if
\[
    u\in \mathrm{USC}(\Omega\times(t_1,t_2])
    \cap C_{\mathrm{loc}}(t_1,t_2;L^{p-1}_{sp}(\mathbb R^d))
\]
and whenever \(\phi\in C^{2}(Q_r(x_0,t_0))\), with
\(Q_r(x_0,t_0)\subset \Omega\times(t_1,t_2]\), touches \(u\) from above at
\((x_0,t_0)\) in \(Q_r(x_0,t_0)\), one has
\[
    \min\left\{
        u(x_0,t_0)-\varphi(x_0,t_0),\,
        \partial_t\phi(x_0,t_0)+(-\Delta_p)^s\phi_r(x_0,t_0)
    \right\}
    \leq 0,
\]
where
\[
    \phi_r(x,t)
    =
    \begin{cases}
        \phi(x,t),
        & (x,t)\in Q_r(x_0,t_0),\\[2mm]
        u(x,t),
        & \text{otherwise}.
    \end{cases}
\]
A viscosity supersolution is defined similarly. Furthermore, a viscosity solution is both a viscosity subsolution and a supersolution.
\end{definition}

The following result will be useful to obtain the comparison principle.

\begin{lemma}\label{l:pointwise-equation}
Let \(u\) be a viscosity subsolution of \eqref{eq:parabolic-obstacle-problem} in \(\Omega \times (t_1,t_2]\). Assume that \(\phi \in C^2(Q_r(x_0,t_0))\) with \((x_0,t_0) \not \in \Lambda_\varphi(u)\), where
\[
    Q_r(x_0,t_0) \subset \Omega \times (t_1,t_2],
\]
and that $\phi$ touches \(u\) from above at \((x_0,t_0)\) in \(Q_r(x_0,t_0)\). Then
\[
    \partial_t \phi(x_0,t_0) + (-\Delta_p)^s u(x_0,t_0) \le \,0.
\]
The supersolution side holds without the requirement that \((x_0,t_0) \not \in \Lambda_\varphi(u)\).
\end{lemma}

\begin{proof}
The argument follows the proof of \cite[Lemma~3.1]{JSU26}. The only additional point in the subsolution case is that we must assume \((x_0,t_0) \not \in \Lambda_\varphi(u)\), since the equation holds only there. The supersolution case follows directly from \eqref{eq:parabolic-obstacle-problem}, which yields
\[
    \partial_t u+(-\Delta_p)^s u \geq 0 \quad \text{in } Q_1.
\]
\end{proof}

We next establish the comparison principle for viscosity supersolutions and subsolutions to \eqref{eq:parabolic-obstacle-problem}. 

\begin{theorem}\label{t:comparison-viscosity}
Let \(u\) and \(v\) be a viscosity supersolution and a viscosity subsolution, respectively, of \eqref{eq:parabolic-obstacle-problem} in \(\Omega \times (t_1,t_2]\). Assume, in addition, that 
\[
 -u,v \in \mathrm{USC} (\overline{\Omega}\times[t_1,t_2]).
\]
If
\begin{align*}
    &u\geq v \text{ on } \p\Omega\times(t_1,t_2] \text{ and on } \Omega \times \{t_1\}\\
    \intertext{and}
    &u \ge v, \quad \text{a.e. in } \bigl(\R^d\setminus \Omega \bigr) \times (t_1,t_2],
\end{align*}
then
\[
    u \ge v \quad \text{in } \Omega \times (t_1,t_2].
\]
\end{theorem}

\begin{proof}
We argue in a similar way to \cite[Theorem 3.2]{JSU26}. We will only describe the setup that allows us to reduce the situation to the one already treated there.

Set
\[
    M\coloneqq \sup_{\Omega\times(t_1,t_2]}(v-u).
\]
Assume, seeking a contradiction, that \(M>0\). By the parabolic boundary ordering, there exists
\[
    (\bar x,\bar t)\in\Omega\times(t_1,t_2]
\]
such that
\[
    M = v(\bar x, \bar t)-u(\bar x, \bar t)>0.
\]
Note that since $u\geq \varphi$, we have \(v(\bar x, \bar t)>\varphi(\bar x, \bar t)\). For \(\varepsilon>0\), define
\[
    M_\varepsilon
    \coloneqq
    \sup_{\substack{x,y\in\overline\Omega\\ t,\tau\in[t_1,t_2]}}
    \left[
        v(x,t)-u(y,\tau)
        -\frac{|x-y|^2}{2\varepsilon}
        -\frac{|t-\tau|^2}{2\varepsilon}
    \right].
\]
Let $(x_\varepsilon,y_\varepsilon,t_\varepsilon,\tau_\varepsilon)$ be a point at which \(M_\varepsilon\) is attained. It then follows that $M_\varepsilon\ge M>0$. After passing to a subsequence,
\[
    x_\varepsilon,y_\varepsilon\to x^\ast,
    \qquad
    t_\varepsilon,\tau_\varepsilon\to t^\ast,
\]
and
\[
    v(x_\varepsilon,t_\varepsilon) - u(y_\varepsilon,\tau_\varepsilon) \to M.
\]
In particular, for \(\varepsilon\) small enough,
\[
    x_\varepsilon,y_\varepsilon\in\Omega,
    \qquad
    t_\varepsilon,\tau_\varepsilon>t_1.
\]
Moreover, the same argument as before gives $v(x^\ast,t^\ast) > \varphi(x^\ast,t^\ast)$, which implies that $v(x_\varepsilon,t_\varepsilon)>\varphi(x_\varepsilon,t_\varepsilon)$ for $\varepsilon$ sufficiently small.

Define
\[
    \phi_\varepsilon(x,t)
    \coloneqq
    u(y_\varepsilon,\tau_\varepsilon)
    +M_\varepsilon
    +\frac{|x-y_\varepsilon|^2}{2\varepsilon}
    +\frac{|t-\tau_\varepsilon|^2}{2\varepsilon},
\]
and
\[
    \psi_\varepsilon(y,\tau)
    \coloneqq
    v(x_\varepsilon,t_\varepsilon)
    -M_\varepsilon
    -\frac{|y-x_\varepsilon|^2}{2\varepsilon}
    -\frac{|\tau-t_\varepsilon|^2}{2\varepsilon}.
\]
Then \(\phi_\varepsilon\) touches \(v\) from above at \((x_\varepsilon,t_\varepsilon)\), and \(\psi_\varepsilon\) touches \(u\) from below at \((y_\varepsilon,\tau_\varepsilon)\). Therefore, by
Lemma \ref{l:pointwise-equation}, we have
\[
    \partial_t\phi_\varepsilon(x_\varepsilon,t_\varepsilon)
    +
    (-\Delta_p)^s v(x_\varepsilon,t_\varepsilon)
    \le0,
\]
and
\[
    \partial_\tau\psi_\varepsilon(y_\varepsilon,\tau_\varepsilon)
    +
    (-\Delta_p)^s u(y_\varepsilon,\tau_\varepsilon)
    \ge0.
\]
From here, we can argue as in \cite[Theorem 3.2]{JSU26} since we arrived at the same viscosity inequalities.
\end{proof}

\section{Spatial regularity via Ishii--Lions method}\label{s:IL}

In this section, we adapt the parabolic Ishii--Lions argument developed in \cite{JSU26} to the equation with an obstacle $\varphi$. The main idea is to reduce to the same viscosity inequalities obtained there, so that we can repeat the exact same argument.

We shall work in the normalized setting
\begin{equation}\label{assumption:normalized-setting}
    \|u\|_{L^\infty(Q_1)} + \sup_{t \in (-1,0]}\|u(\cdot,t)\|_{L^{p-1}_{sp}(\R^d)} \leq 1.
\end{equation}

Moreover, we shall use the following bootstrap hypothesis. For some
\(\kappa\in[0,1)\), assume that there exists a constant \(C_1(\kappa)\) such that, for every \(x,y\in B_1\) and \(t\in(-1,0]\),
\begin{equation}\label{assumption:bat-ind}
    |u(x,t)-u(y,t)|
    \leq
    C_1(\kappa)|x-y|^\kappa.
\end{equation}
To start the bootstrapping argument, we begin by assuming only that $u$ is uniformly continuous in \(\overline Q_1\) and therefore that this condition is satisfied for $\kappa=0$.

\subsection{The setup}

For constants $L, L_2, \beta^\ast$ to be fixed, and $t_0 \in (-1,0]$, define the function with doubled spatial variables
\[
    \Phi(x,y,t) \coloneqq u(x,t)-u(y,t)-L\,\omega(\abs{x-y})-L_2 \,\psi(x) - L_2(t_0-t)^{1+\beta^\ast},
\]
for $x,y \in B_1$ and $t \in (-1,t_0]$. We will consider two different moduli of continuity $\omega$, the first of H\"older type and the second of Lipschitz type, modified by a lower order term to make it strictly concave. We denote them respectively by
\begin{align*}
    \omega_\gamma(r) \coloneqq r^\gamma,
\end{align*}
and
\begin{align*}
    \widetilde\omega(r) \coloneqq r + \frac{r}{20 \log (r/4)}.
\end{align*}
In either case, it will always satisfy
\begin{equation}\label{eq:omega-linear-lower-bound}
    \omega(r)\geq c_\omega r
\end{equation}
for some \(c_\omega>0\) and $r\leq r_0$ universal. The localizing function is given by $\psi(x)\coloneqq \psi_0(x)^m$, $m>2$, for a smooth nonnegative $\psi_0$ that vanishes on $\overline{B}_{1/2}$ and is strictly positive on $B_1\setminus \overline{B}_{1/2}$. This choice implies that there exists a constant $C>0$ such that
\[
    |\nabla \psi(x)|\leq C\psi(x)^\frac{m-1}{m}, \quad \text{ for all } x\in B_1.
\]
We will prove that, for a choice of $L_2$, $m$, and $\beta^\ast$ large and universal, there exists $L$ large enough such that the function $\Phi$ is nonpositive, which implies that $u$ has modulus of continuity $\omega$ in space at $t = t_0$. For the sake of simplicity, we assume $t_0 = 0$.

Suppose, seeking a contradiction, that
\[
    M_L
    \coloneqq
    \max_{\substack{x,y\in\overline{B_1}\\t\in[-1,t_0]}}
    \Phi(x,y,t)>0,
\]
and let \((\bar x,\bar y,\bar t)\) be a maximum point. We first claim
that
\begin{equation}\label{eq:bar-x-outside-coincidence}
    u(\bar x,\bar t)>\varphi(\bar x,\bar t).
\end{equation}
Indeed, since \(M_L>0\) and the localization terms are nonnegative,
\begin{equation}\label{eq:positive-max-obstacle}
    u(\bar x,\bar t)-u(\bar y,\bar t)
    >
    L\omega(\abs{\bar x-\bar y}).
\end{equation}
In particular, \(\bar x\neq\bar y\). We will now observe that \((\bar x,\bar t)\not\in\Lambda_\varphi(u)\). If, on the contrary,
 \(u(\bar x,\bar t)=\varphi(\bar x,\bar t)\) then
\[
\begin{aligned}
    u(\bar x,\bar t)-u(\bar y,\bar t)
    &\leq
    \varphi(\bar x,\bar t)-\varphi(\bar y,\bar t)\\
    &\leq [\varphi]_{C^{0,1}_x}\abs{\bar x-\bar y}.
\end{aligned}
\]
On the other hand, by \eqref{eq:omega-linear-lower-bound} and
\eqref{eq:positive-max-obstacle},
\[
    u(\bar x,\bar t)-u(\bar y,\bar t)
    >
    c_\omega L\abs{\bar x-\bar y}.
\]
This is a contradiction as soon as
\begin{equation}\label{eq:L-obstacle-choice}
    L>\frac{[\varphi]_{C^{0,1}_x}}{c_\omega}.
\end{equation}
Hence \eqref{eq:bar-x-outside-coincidence} holds.

We now double the time variable and define
\[
\begin{aligned}
\Psi_K(x,y,t,\tau)\coloneqq{}&
u(x,t)-u(y,\tau)-L\,\omega(\abs{x-y})-L_2\psi(x) \\
&\qquad
-L_2(t_0-t)^{1+\beta^\ast}-K(t-\tau)^2.
\end{aligned}
\]
Let
\begin{equation}\label{eq:MK-geq-ML-obstacle}
    M_K
    \coloneqq
    \max_{\substack{x,y\in\overline{B_1}\\
t,\tau\in[-1,t_0]}}
\Psi_K(x,y,t,\tau) = \Psi_K(x_K,y_K,t_K,\tau_K) \geq M_L>0.
\end{equation}
We moreover define
\[
    a_K=x_K-y_K,\quad b_K \coloneqq t_K-\tau_K.
\]
This implies in particular that, up to subsequence, $a_K\to \tilde a$ as $K\to \infty$ for some $\tilde a$ satisfying $|\tilde a|>0$.

The usual doubling-of-time argument yields
\begin{equation}\label{eq:time-collapse-obstacle}
    \abs{t_K-\tau_K}\longrightarrow0,
    \qquad
    K\abs{t_K-\tau_K}^2\longrightarrow0,
    \qquad
    M_K\longrightarrow M_L
\end{equation}
as \(K\to\infty\).

We claim that
\begin{equation}\label{eq:xK-outside-coincidence}
    u(x_K,t_K)>\varphi(x_K,t_K)
\end{equation}
for every sufficiently large \(K\). Otherwise, there would exist a
sequence \(K_j\to\infty\) such that
\[
    u(x_{K_j},t_{K_j})
    =
    \varphi(x_{K_j},t_{K_j}).
\]
Passing to a subsequence, we may assume that
\[
\begin{aligned}
    x_{K_j}&\longrightarrow x_\infty,
    &y_{K_j}&\longrightarrow y_\infty,\\
    t_{K_j}&\longrightarrow t_\infty,
    &\tau_{K_j}&\longrightarrow t_\infty.
\end{aligned}
\]
By the continuity of \(u\) and \(\varphi\),
\begin{equation}\label{eq:limit-contact-obstacle}
    u(x_\infty,t_\infty)
    =
    \varphi(x_\infty,t_\infty).
\end{equation}
Moreover, \eqref{eq:time-collapse-obstacle} implies
\[
    \Phi(x_\infty,y_\infty,t_\infty)
    =
    M_L>0.
\]
Thus \((x_\infty,y_\infty,t_\infty)\) is a positive maximum point of
\(\Phi\). The argument leading to
\eqref{eq:bar-x-outside-coincidence}, applied to this maximum point,
gives
\[
    u(x_\infty,t_\infty)
    >
    \varphi(x_\infty,t_\infty),
\]
contradicting \eqref{eq:limit-contact-obstacle}. This proves
\eqref{eq:xK-outside-coincidence}.

Consequently, for \(L\) satisfying \eqref{eq:L-obstacle-choice} and
\(K\) sufficiently large, the point
\((x_K,t_K)\) does not belong to the coincidence set. Hence, the viscosity subsolution inequality for
\[
    \partial_tu+(-\Delta_p)^su=0
\]
holds at \((x_K,t_K)\). The viscosity supersolution inequality
holds at \((y_K,\tau_K)\), whether or not
\((y_K,\tau_K)\) belongs to the coincidence set.

In the argument of this section, we will consider $L$ to be a constant fixed large enough to satisfy some conditions which will appear throughout the section, but which is independent of $K$. Then, we take $K$ large enough so that $|a_K|>|\tilde a|/2$. Note that $\tilde a$ depends implicitly on $L$, which is fixed. Since $b_K\to 0$, we can now take $K$ possibly even larger to ensure
\begin{equation}\label{eq:rusbad}
    \begin{aligned}
     &\left(\displaystyle \int_{B_{\frac{1}{16}}^c} |u(y_K+z, t_K)-u(y_K+z, \tau_K)|^{p-1}\frac{\dd z}{|z|^{d+sp}}\right)^\frac{1}{p-1}\\
    &\qquad \leq \frac{|\tilde a|}{2} \leq |a_K|.
\end{aligned}
\end{equation}
Therefore, $K$ depends on $L$, but this dependence is harmless. Moreover, $K$ depends on the modulus of continuity of the $\Tail$ of $u$. But since $L$ does not depend on $K$, this dependence is harmless as well.

By construction, the points $t_K$ and $\tau_K$ are interior.  Therefore, for $K$ sufficiently large, we can assume that
\begin{align}\label{eq:rusbem}
    |u(z,t_K)-u(z,\tau_K)|\leq \frac{|\tilde a|}{2}\leq |a_K|,
\end{align}
for any $z\in \overline B_1$.

Define the auxiliary function 
\begin{equation*}
    \phi(x,y) \coloneqq L\omega(|x-y|) + L_2 \psi(x).
\end{equation*}
To obtain viscosity inequalities, we define the test functions 
\begin{align*}
w_1(x,t)
&=
\begin{cases}
\Lambda_1 + \phi(x,y_K) + L_2|t|^{1+\beta^\ast} + K(t-\tau_K)^2
& \text{for } x\in B_{\delta_1\,|a_K|}(x_K),\\
u(x,t)
& \text{otherwise},
\end{cases}
\\[1em]
w_2(y,\tau)
&=
\begin{cases}
\Lambda_{2} - \phi(x_K,y) - K(t_K-\tau)^2
& \text{for } y\in B_{\delta_1\,|a_K|}(y_K),\\
u(y,\tau)
& \text{otherwise},
\end{cases}
\end{align*}
where $\Lambda_{1} = u(y_K,\tau_K) + M_K$ and $\Lambda_{2} = u(x_K,t_K) - M_K -L_2|t_K|^{1+\beta^\ast}$. The constant $\delta_1$ will be chosen to be small, possibly depending on $|a_K|$. The time variables $t$ and $\tau$ are defined in a neighborhood of $t_K$ and $\tau_K$, respectively.

Note that $w_1$ touches $u$ from above at $(x_K,t_K)$ and $w_2$ touches $u$ from below at $(y_K,\tau_K)$. Therefore, by the viscosity condition, we have 
\begin{align*}
     &\p_t w_1(x_K,t_K) + (-\Delta_p)^s w_1(x_K,t_K) \leq \, 0,\\
     &\p_t w_2(y_K,\tau_K) + (-\Delta_p)^s w_2(y_K, \tau_K) \geq \, 0.
\end{align*}
Subtracting the inequalities, we obtain 
\begin{equation}\label{eq:evaluate-at-max}
   \begin{aligned}
       \II\,&\coloneqq(-\Delta_p)^s w_1(x_K,t_K) - (-\Delta_p)^s w_2(y_K, \tau_K) \\
       &\leq \p_tw_2(y_K,\tau_K)-\p_tw_1(x_K,t_K)= L_2(1+\beta^\ast)(-t_K)^{\beta^\ast}.
   \end{aligned}
\end{equation}

For the sake of organization, we will split the analysis of the term $\II$ and the time derivative part into several lemmas, and combine them at the end of this section to get a contradiction, in the different scenarios, with the fact that we are assuming \eqref{eq:MK-geq-ML-obstacle}.

\subsection{Geometric estimates}

In this section, we estimate each term appearing from the viscosity inequalities from \eqref{eq:evaluate-at-max}. To bound $\II$, we split the spatial ball into four appropriate domains and bound each of them accordingly. We start by bounding the right-hand side.

\begin{lemma}\label{l:estimate_time}
Let $u$ be a viscosity solution to \eqref{eq:parabolic-obstacle-problem}, and assume \eqref{assumption:normalized-setting} and \eqref{assumption:bat-ind}. Let $K$ be large enough so that \eqref{eq:rusbad} holds.  Then, the quantity $\II$ defined in \eqref{eq:evaluate-at-max} satisfies
\begin{equation*}
        \II \leq C(\kappa)|a_K|^\frac{\kappa \beta^\ast}{1+\beta^\ast},
\end{equation*}    
where $C>0$ depends on $L_2$, $\beta^\ast$, $\kappa$, and universal constants.
\end{lemma}

\begin{proof} 
From the previous discussion, we know that the point where the supremum \eqref{eq:MK-geq-ML-obstacle} is attained is interior. Since the supremum in \eqref{eq:MK-geq-ML-obstacle} is positive, we have
\[
    L_2(-t_K)^{1+\beta^\ast} < u(x_K,t_K) - u(y_K,\tau_K).
\]
Combining it with assumptions \eqref{assumption:bat-ind} and \eqref{eq:rusbem}, we obtain
\[
    L_2(-t_K)^{1+\beta^\ast} < C(\kappa)|a_K|^\kappa +C_1|a_K|.
\]
In view of \eqref{eq:evaluate-at-max}, we obtain
\begin{align*}
   \II & \leq L_2(1+\beta^\ast)|t_K|^{\beta^\ast}\\
   & = (1+\beta^\ast)L_2^{\frac{1}{1+\beta^\ast}} C^{\frac{\beta^\ast}{1+\beta^\ast}} \left(C(\kappa)|a_K|^\kappa \right)^\frac{\beta^\ast}{1+\beta^\ast},
\end{align*}
as intended.
\end{proof}

Recall that we denoted $a_K=x_K-y_K$.
Let $\mathcal{C} = \mathcal{C}(a_K)$ be the cone of directions defined by 

\begin{equation*}
    \C(a_K) \coloneqq \left\{ z \in B_{|a_K|/2} \colon |\langle a_K, z \rangle| \geq \sqrt{1 - \delta_0^2} \, |a_K|\,|z| \right\}.
\end{equation*}
Let also
\[
    \mathcal{D}_1 = B_{\delta_1\,|a_K|} \cap \mathcal{C}^c,
    \qquad
    \mathcal{D}_2 = B_{1/16} \setminus (\mathcal{D}_1 \cup \mathcal{C}).
\]
Here $\delta_1$ is a small quantity to be chosen later, depending possibly on $|a_K|$. We will assume that $\delta_1\ll \delta_0$.

Given $D \subset \R^d$, we introduce the useful notation
\begin{equation*}
    \begin{aligned}
    \LL[D]w(x)  \coloneqq  \int_D \bigl|w(x)-w(x+z)\bigr|^{p-2}\bigl(w(x)-w(x+z)\bigr)\,|z|^{-d-sp}\dd z.
\end{aligned}
\end{equation*}
We then write the quantity $\II$ from \eqref{eq:evaluate-at-max} as 
\begin{eqnarray*}
   \II &  = &\LL[\mathcal{C}]w_1(x_K,t_K)-\LL[\mathcal{C}]w_2(y_K,\tau_K)\\
    & &+ \LL[\mathcal{D}_1]w_1(x_K,t_K)-\LL[\mathcal{D}_1]w_2(y_K,\tau_K)\\
     & &+  \LL[\mathcal{D}_2]w_1(x_K,t_K)-\LL[\mathcal{D}_2]w_2(y_K,\tau_K)\\
     & &+ \LL[B_{1/16}^c]w_1(x_K,t_K)-\LL[B_{1/16}^c]w_2(y_K,\tau_K)\\
     & \eqqcolon &\II_1 + \II_2 + \II_3 + \II_4.
\end{eqnarray*}

We recall the following estimates in the concavity cone; see \cite[Lemmas 4.4 and 4.5]{JSU26}.

\begin{lemma}\label{l:estimate:l1l2}
Let $(x_K, y_K, t_K,\tau_K)$ satisfy \eqref{eq:MK-geq-ML-obstacle}. Then, for $L$ sufficiently large, the following estimates hold:
\[
    \II_1 \geq c\bigl(L\omega'(|a_K|)\bigr)^{p-2}\int_{\C} |z|^{p-2-d-sp}\delta^2\phi(\cdot, y_K)(x_K,z)\dd z,
\]
and
\[
    \II_2 \geq -C \bigl(L\omega'(|a_K|)\bigr)^{p-1} \delta_1^{p(1-s)}|a_K|^{p(1-s)-1}.
\]
The constants $c$, $C$, and $L$ are universal.
\end{lemma}

We now bound the term in $\mathcal{D}_2$.

\begin{lemma}[Estimate in $\mathcal{D}_2$]\label{l:estimate_I3}
Suppose that $u$ satisfies \eqref{assumption:bat-ind} for some
\[
    \kappa\in[0,\min\{sp/(p-2),1\}).
\]
Let $K$ be large enough to satisfy \eqref{eq:rusbad}. Then, for any $\theta\in(0,1)$ and for any $L$ large enough universal, we have
\begin{equation}\label{eq:estimate_I3}
    \begin{aligned}
\II_3 \ge -C  \biggl[
&\int_{\delta_1|a_K|}^{|a_K|^{\theta}} r^{\kappa(p-2)+1-sp}\,\dd r
\allowbreak + |a_K|^{\frac{m-1}{m}\kappa}
\int_{\delta_1|a_K|}^{|a_K|^{\theta}} r^{\kappa(p-2)-sp}\,\dd r  \\
&\allowbreak + |a_K|^{\kappa+\theta(\kappa(p-2)-sp)}
\biggr].
\end{aligned}
\end{equation}
Here, $C$ depends on $\kappa, [u]_{C_x^{0,\kappa}}$ and universal constants.
\end{lemma}

\begin{proof}
The proof is identical to that of \cite[Lemma 4.6]{JSU26}. We point out that the proof there relies on H\"older continuity in the time variable; in the present setting, this is replaced by \eqref{eq:rusbem}. The argument remains valid when $\kappa=0$.
\end{proof}

Finally, we control the term involving the tail. Here, we use Assumption \eqref{eq:rusbad} in a crucial way. 

\begin{lemma}[Estimate on $B_{1/16}^c$]\label{l:estimate_I4}
Suppose $u$ satisfies Assumption \eqref{assumption:bat-ind} for some $\kappa\in(0,1]$ and let $K$ be large enough to satisfy \eqref{eq:rusbad}. For $L$ large enough, we have
\begin{align*}
    \II_4\geq -C |a_K|^\kappa,
\end{align*}
where  $C$ depends on universal constants and $[u]_{C^{\kappa}_x}$.
\end{lemma}

\begin{proof}
The proof is identical to that of \cite[Lemma 4.7]{JSU26}. When $\kappa=0$, the argument simplifies, since it suffices to show that $\II_4$ is bounded by a constant. As in the previous estimate, the H\"older continuity estimate in time is replaced here by \eqref{eq:rusbem}.
\end{proof}

We collect the previous estimates and prove the Lipschitz regularity in the spatial variable. 

\begin{theorem}[Lipschitz in space]\label{t:lip-in-space_Chibrata}
Let $u$ be a viscosity solution to \eqref{eq:parabolic-obstacle-problem} in $Q_1$. Suppose further that $u \in C_{\rm loc}(-1,0;L^{p-1}_{sp}(\R^d))$ and that \eqref{assumption:normalized-setting} and \eqref{assumption:bat-ind} are in force. Then, $u$ is Lipschitz in space at $t_0$ and satisfies
\[
    |u(x,t_0)-u(y,t_0)|\leq C|x-y|, \quad \text{for all }  x,y \in B_{1/2}.
\]
The constant \(C\) depends only on universal parameters, and \([\varphi]_{C^{0,1}_x}\).
\end{theorem}

\begin{proof}
The proof follows the same steps as \cite[Theorem 4.1]{JSU26}. 

Denote $\bar \gamma \coloneqq \min\{1,sp/(p-1)\}$. For each fixed $L$, we take $K$ large enough to satisfy \eqref{eq:rusbad} and \eqref{eq:rusbem}.

The proof is split into the following five steps by combining Lemmas \ref{l:estimate_time}, \ref{l:estimate:l1l2}, \ref{l:estimate_I3}, and \ref{l:estimate_I4}, assuming that $t_0=0$ in each step. 

\begin{enumerate}
    \item We show that $u\in C^{0,\bar \gamma^-}_x$ locally,
    \item We show that $u\in C^{0,\bar \gamma}_x$ locally when $\bar \gamma<1$,
    \item We show that $u\in C^{0,1}_x$ locally when $\bar \gamma = 1$,
    \item We show that $u\in C^{0,1^-}_x$ locally when $\bar \gamma<1$,
    \item We show that $u\in C^{0,1}_x$ locally when $\bar \gamma < 1$.
\end{enumerate}

In the first step, when $\kappa=0$, the same argument as in \cite{JSU26} still applies. The choice of $m$ to apply estimate \eqref{eq:estimate_I3} in Lemma \ref{l:estimate_I3} is not important in this case. The remaining steps are identical.
\end{proof}

\section{Time regularity and proof of the main theorem}\label{s:time:reg}

In this section, we prove the regularity result with respect to time and conclude with the proof of the main theorem. 

\begin{proposition}\label{p:time-regularity-estimate-phip-desc}
Assume $u$ is a viscosity solution to \eqref{eq:parabolic-obstacle-problem} in $Q_1$ satisfying \eqref{assumption:normalized-setting}. In addition, suppose that \eqref{assumption:bat-ind} holds with $\kappa = 1$. Then, there exists a constant $C>0$ such that
\[
    \sup_{x \in B_{1/2}}|u(x,t_1) - u(x,t_2)| \leq C |t_1-t_2|^{\alpha},
\]
for every $t_1,t_2 \in (-1/4,0]$.  The exponent $\alpha$ is given by
\begin{align*}
    \alpha = \min\left\{ \left( \dfrac{1}{1 - q_c}\right)^-,\alpha_\varphi \right\},
\end{align*}
where $q_c \coloneqq - 1+p(1-s)$.  The constant \(C\) depends only on universal parameters, $\alpha$, \([\varphi]_{C^{0,1}_x}\), and \([\varphi]_{C^{0,\alpha_\varphi}_t}\).

In particular, $u$ is Lipschitz in time when $q_c>0$ and $\alpha_\varphi=1$.
\end{proposition}

\begin{proof}
Fix $t_0 \in (-1/4,0)$ and define
\[
    \Psi_\eta(x,t) \coloneqq \eta + L_0(t-t_0) + L_1|x|^\gamma+A(t-t_0)^{\alpha_\varphi},
\]
for $\eta>0$ to be chosen. The constant $A$ only needs to satisfy $A>[\varphi]_{C^{0,\alpha_\varphi}_t}$. The parameters $L_0$ and $L_1$ depend on universal quantities, $[\varphi]_{C^{0,1}_x}$, $\eta$, and
\[
    \gamma = \max\left\{\left(1 - \frac{q_c}{p-1} \right)^+,1\right\}.
\]
These choices ensure that the local part of $(-\Delta_p)^s\Psi_\eta$ is finite.

Let us compare the values of the difference $u(x,t) - u(0,t_0)$ with $\Psi_\eta$ on the region $(\overline{B_1} \setminus B_{1/2}) \times (t_0,0]$ and $B_1 \times \{t_0\}$. For $x \in \overline{B_1} \backslash B_{1/2}$, we have
\[
    u(x,t) - u(0,t_0) \leq 2\|u\|_{L^\infty(Q_1)}\leq 2,
\]
and so for $L_1 \geq 2^{1+\gamma}$, we have
\[
    u(x,t) - u(0,t_0) \leq \Psi_\eta(x,t).
\]
For $x \in B_1$ and $t=t_0$, we prove that by making $L_1$ large enough, we have
\[
    u(x,t_0) - u(0,t_0) \leq \Psi_\eta(x,t_0) = \eta + L_1|x|^\gamma.
\]
Indeed, since $u$ is Lipschitz in space, we know that there is a constant $C=C([\varphi]_{C^{0,1}_x})>0$ such that
\[
    |u(x,t_0) - u(0,t_0)| \leq C|x|, \quad \text{for }  x \in B_{1/2},
\]
and thus we only need to pick $L_1$ such that
\[
    C|x| \leq \eta + L_1|x|^\gamma, \quad \text{for } x \in B_{1/2}.
\]
If we choose
\[
    L_1 \geq C\left(\frac{\eta}{C} \right)^{1-\gamma},
\]
then we have, for $|x| > \eta/C$, that
\[
    C|x| = C |x|^{1 - \gamma}|x|^\gamma < C \left(\frac{\eta}{C}\right)^{1-\gamma}|x|^\gamma.
\]
On the other hand, for $|x|\leq \eta/C$, we directly have
\[
   C|x| \leq \eta \leq \eta + L_1|x|^\gamma. 
\]
Therefore, the constant $L_1$ has to satisfy
\[
    L_1 \geq \max\left\{2^{1+\gamma}, C\left(\frac{\eta}{C} \right)^{1-\gamma} \right\},
\]
and so we choose $L_1 \geq C\eta^{1-\gamma}$, for $C$ large enough, depending on universal quantities and $[\varphi]_{C^{0,1}_x}$. Note that the above computations simplify considerably when $\gamma=1$, but the same conclusions still hold. We have shown that
\begin{align}\label{eq:compa_psi}
    u- u(0,t_0) \leq \Psi_\eta \quad \text{on} \quad \left(\overline{B}_1 \backslash B_{1/2}\right) \times (t_0,0] \quad \text{and} \quad B_1 \times \{ t_0\}. 
\end{align}
Defining 
\begin{align*}
    v(x,t) \coloneqq \begin{cases}
     \Psi_\eta(x,t) & \text{ for } (x,t) \in B_1 \times (t_0,0],\\
    u(x,t)- u(0,t_0) & \text{ otherwise},
\end{cases}
\end{align*}
it follows from \eqref{eq:compa_psi} that
\[
    u- u(0,t_0) \leq v  \quad \text{on} \quad \left(\R^d \backslash B_{1/2}\right) \times (t_0,0] \quad \text{and} \quad B_1 \times \{t_0\}.
\]
We wish to use the comparison principle to prove that this inequality propagates to the interior cylinder $B_{1/2} \times (t_0,0]$. To do so, we recall the computations in \cite[Proposition 5.1]{JSU26} in which we showed that
\[
    (-\Delta_p)^s v \geq -CL_1^{p-1} \quad \text{in} \quad B_{1/2} \times (t_0,0].
\]
Therefore, choosing $L_0=C_1 L_1^{p-1}$ where $C_1$ is large, universal, we have
\[
    \p_t \Psi_\eta + (-\Delta_p)^s v \geq L_0 - CL_1^{p-1} + A \alpha_\varphi(t-t_0)^{\alpha_\varphi - 1} > 0,
\]
in $B_{1/2} \times (t_0,0]$. We also show that $\Psi_\eta$ lies strictly above the translated obstacle $\varphi - u(0,t_0)$. Indeed,
\begin{align*}
    \varphi(x,t) - u(0,t_0) & \leq \varphi(x,t) - \varphi(0,t_0)\\
    & \leq [\varphi]_{C^{0,1}_x}|x| + [\varphi]_{C^{0,\alpha_\varphi}_t}(t-t_0)^{\alpha_\varphi}\\
    & < \Psi_\eta(x,t),
\end{align*}
where the last inequality follows from $A>[\varphi]_{C^{0,\alpha_\varphi}_t}$ as well as finally choosing the parameter $L_1$ by
\[
L_1=L_1(\eta)=C(\gamma,[\varphi]_{C^{0,1}_x})\eta^{1-\gamma}.
\]
By the comparison principle (Theorem~\ref{t:comparison-viscosity}), we have
\[
    u(x,t) - u(0,t_0) \leq \Psi_\eta(x,t) \quad \text{on} \quad B_{1/2} \times (t_0,0].
\]
In particular, we have, for any $t \in (t_0,0]$,
\[
    u(0,t) - u(0,t_0)  \leq \eta + CL_1^{p-1}(t-t_0) + A(t-t_0)^{\alpha_\varphi},
\]
for any given $\eta>0$. By the choice of $L_1$, we have
\[
    u(0,t) - u(0,t_0) \leq \eta + C\eta^{(1-\gamma)(p-1)}(t-t_0) + A(t-t_0)^{\alpha_\varphi}, 
\]
for any $t \in (t_0,0]$ and any $\eta \in (0,1)$. Since, for each fixed $t$, this estimate holds for all $\eta$, we can optimize in the parameter $\eta$ in the following way. If $\gamma=1$, then we can just take $\eta=0$. Otherwise, call
\[
    f(\eta) = \eta+c\eta^{(1-\gamma)(p-1)},
\]
with $c = C(t-t_0)$. We see that $f'(\eta)=0$ when
\[
    \eta = \left(c(\gamma-1)(p-1)\right)^\frac{1}{(\gamma-1)(p-1)+1},
\]
and so, we reach
\begin{align*}
    u(0,t) - u(0,t_0)& \leq C(t-t_0)^{\frac{1}{1+(p-1)(\gamma-1)}}+ A(t-t_0)^{\alpha_\varphi}\\
    &\leq C(t-t_0)^\alpha,
\end{align*}
as intended. The lower bound is obtained by considering $-\Psi_\eta$ and observing that it is a strict subsolution. It does not need to lie above the obstacle. 
\end{proof}

We now present the proof of Theorem \ref{t:Lipschitz}.

\begin{proof}[Proof of Theorem \ref{t:Lipschitz}]
Fix $(x_0,t_0)\in Q_{1/2}$. Let $\lambda>0$, to be chosen below, be small and set
\[
    \mathcal{M}_1\coloneqq C_1\mathcal M,
\]
where $\mathcal M$ is defined in \eqref{eq:constant-M} and
$C_1=C_1(d,s,p,\lambda)>0$ is chosen sufficiently large. Define
\[
    v(x,t)\coloneqq \mathcal{M}_1^{-1}
    u\!\left(\lambda x+x_0,\,
    \mathcal{M}_1^{\,2-p}\lambda^{sp}t+t_0\right).
\]
Then $v$ satisfies Assumption~\eqref{assumption:normalized-setting}. 

We now note that if $\lambda$ is chosen sufficiently small, in a universal way, then $v$ satisfies the same equation \eqref{eq:parabolic-obstacle-problem} in $Q_4$, with the obstacle replaced by 
$$\widetilde \varphi(x,t)\coloneqq \mathcal{M}_1^{-1}\varphi(\lambda x+x_0,\mathcal{M}_1^{2-p}\lambda^{sp}t+t_0).$$ 
Note that $[\widetilde\varphi]_{C_x^{0,1}}\leq [\varphi]_{C_x^{0,1}} $ and $[\widetilde\varphi]_{C^{0,\alpha_\varphi}_t}\leq [\varphi]_{C^{0,\alpha_\varphi}_t}$.

We apply Theorem \ref{t:lip-in-space_Chibrata} to obtain
\[
    |v(x,0)-v(y,0)|\le C|x-y|, \quad \text{for } x,y\in B_{1/2}.
\]
We now invoke Proposition \ref{p:time-regularity-estimate-phip-desc} to get
\[
    \sup_{x\in B_{1/4}}|v(x,t)-v(x,0)|\le C|t|^\alpha, \quad \text{for } t\in(-1/4,0],
\]
where $\alpha$ is as in Proposition \ref{p:time-regularity-estimate-phip-desc}. Combining the previous estimates, we conclude that
\[
    |v(x,t)-v(0,0)|\le C(|x|+|t|^\alpha), \quad \text{for } x\in B_{1/2}\text{ and } t\in(-1/4,0].
\]
Rescaling back to $u$ gives 
\[
    |u(x,t) - u(x_0,t_0)| \leq C\mathcal{M}_1\left(|x-x_0|+\mathcal{M}_1^{(p-2)\alpha}|t-t_0|^\alpha \right),
\]
for all
\[
    (x,t)\in B_\lambda(x_0)\times\bigl(t_0-\mathcal{M}_1^{2-p}\lambda^{sp},\,t_0\bigr].
\]
This proves the desired estimate in a (small) cylinder centered at $(x_0,t_0)$. By a covering argument, we obtain
\begin{align*}
    |u(x,t) - u(y,\tau)| & \leq C\mathcal{M}_1\left(|x-y|+\mathcal{M}_1^{(p-2)\alpha}|t-\tau|^\alpha \right),
\end{align*}
for any $(x,t), (y,\tau) \in Q_{1/2}$.
\end{proof}

\medskip

{\small \noindent{\bf Acknowledgments.} This publication is based upon work supported by King Abdullah University of Science and Technology (KAUST) under Award No. ORFS-CRG12-2024-6430.}

\end{document}